\documentclass[a4paper]{amsart}
\usepackage{amsmath,amssymb}
\usepackage[margin=2cm]{geometry}
\usepackage[colorlinks,bookmarks=true,citecolor=blue,linkcolor=blue,urlcolor=blue, breaklinks=true]{hyperref}
\usepackage{cleveref}

\usepackage{bbm}

\usepackage[sort&compress,numbers]{natbib}
\usepackage{doi}

\numberwithin{equation}{section}
\newtheorem{theorem}{Theorem}
\newtheorem{lemma}{Lemma}[section]
\newtheorem{corollary}[lemma]{Corollary}
\newtheorem{proposition}[lemma]{Proposition}
\newtheorem{algorithm}[lemma]{Algorithm}
\theoremstyle{definition}
\newtheorem{definition}[lemma]{Definition}
\theoremstyle{remark}
\newtheorem*{remark}{Remark}
\newtheorem*{remarks}{Remarks}

\DeclareMathOperator{\tr}{tr}
\newcommand{\id}{\mathbbm{1}}
\newcommand{\C}{\mathbb{C}}
\newcommand{\Zzeta}{\mathbb{Z}[\zeta]}
\newcommand{\Zp}{\mathbb{Z}_p}

\title{Generalised Burnside and Dixon algorithms for irreducible projective representations}
\author{Attila Szabó}
\address{Physik-Institut, Universität Zürich, Winterthurerstr.\ 190, 8057 Zürich, Switzerland}
\thanks{A.\,Sz. was supported by Ambizione grant No. 215979 by the Swiss National Science Foundation.}
\email{attila.szabo@physik.uzh.ch}
\date{11 May 2025}

\begin{document}

\begin{abstract}
    Based on the character theory of projective representations of finite groups proposed in Ref.~\cite{Cheng2015CharacterTheory}, we generalise several algorithms for computing character tables and matrices of irreducible linear representations to projective representations. In particular, we present an algorithm based on that of Burnside to compute the characters of all irreducible projective representations of a finite group with a given Schur multiplier, and transpose it to exact integer arithmetic following Dixon's character table algorithm. We also describe an algorithm based on that of Dixon to split a projective representation into irreducible subspaces in floating-point arithmetic, and discuss how it can be used to compute matrices for all projective irreps with a given multiplier. Our algorithms bypass the construction of the representation group of the Schur multiplier, which makes them especially attractive for floating-point computations, where exact values of the multiplier are not necessarily available.
\end{abstract}

\maketitle

\section{Introduction}

Throughout this paper, $G$ is a finite group with identity element $1$.

The standard way to study projective representations of a finite group $G$ proceeds by constructing a larger representation group $G^*$ such that its linear representations correspond to the projective representations of $G$. 
An alternative to this approach is the character theory introduced by Ref.~\cite{Cheng2015CharacterTheory}, which bypasses $G^*$. It is largely analogous to standard character theory and retains most of its attractive features: characters uniquely specify projective representations up to isomorphism, they are essentially class functions, and characters of irreducible representations are an orthonormal basis of the space of such class functions.

The goal of the present paper is to demonstrate how this character theory can be used to algorithmically compute characters and representation matrices of projective irreducible representations without resorting to the representation group $G^*$. In particular, we generalise Burnside's~\cite[\S223]{Burnside1911} and Dixon's~\cite{Dixon1967} algorithms for computing character tables and Dixon's algorithm for computing matrices of irreducible representations~\cite{Dixon1970}.

The main results of the paper are the following three theorems:

\begin{theorem}[Burnside's algorithm for character tables of projective irreps]
    \label{thm: burnside}
    Let $\alpha$ be a unitary Schur multiplier (\Cref{def: multiplier}) on $G$.
    Let $C^{(1)},\dots,C^{(m)}$ be the $\alpha$-regular conjugacy classes of $G$ (cf.~\Cref{thm: regular classes}) and fix one representative element $c^{(i)}_0\in C^{(i)}$ for all $1\le i\le m$. Define the $m\times m$ matrices $M_A$ for each $\alpha$-regular class $A$ as
    \begin{equation}
        (M_A)_{BC} = \frac1{\sqrt{|B||C|}} \sum_{\substack{a\in A \\ b\in B\\ab\in C}}\frac{\beta(c_0,ab)}{\beta(a_0,a)\beta(b_0,b)}\alpha(a,b),
        \label{eq: burnside M matrix}
    \end{equation}
    where $a_0,b_0,c_0$ are the chosen representatives of the $\alpha$-regular classes $A,B,C$ and 
    \begin{equation*}
        \beta(g,g'=hgh^{-1}) := \frac{\alpha(h,h^{-1})}{\alpha(h,gh^{-1})\alpha(g,h^{-1})} = \frac{\alpha(h,h^{-1})}{\alpha(h,g)\alpha(hg,h^{-1})}
    \end{equation*}
    for all $\alpha$-regular conjugate $g,g'\in G$ (cf.~\Cref{thm: beta uniqueness}).
    \begin{enumerate}
        \item Let $\chi$ be the character of an irreducible projective representation of degree $d$ of $G$ with multiplier $\alpha$ (cf.~\Cref{def: irrep}). 
        Then the eigenvalue equation
        \begin{equation}
            M_A \left(\begin{array}{c}
                  \chi\big(c^{(1)}_0\big)\sqrt{|C^{(1)}|/|G|}\\
                 \vdots\\
                  \chi\big(c^{(m)}_0\big)\sqrt{|C^{(m)}|/|G|}\\
            \end{array}\right) 
            = 
            \frac{|A|\chi(a_0)}d \left(\begin{array}{c}
                  \chi\big(c^{(1)}_0\big)\sqrt{|C^{(1)}|/|G|}\\
                 \vdots\\
                  \chi\big(c^{(m)}_0\big)\sqrt{|C^{(m)}|/|G|}\\
            \end{array}\right) 
            \label{eq: burnside eigenvalue equation}
        \end{equation}
        holds for each $\alpha$-regular class $A$, with representative element $a_0$.
        In other words, the vector on the right hand side is a joint eigenvector of all the matrices $M_A$.
        \item The matrices $M_A$ are normal and they commute. Their $m$ joint eigenvectors are uniquely defined up to a scalar factor and correspond to the distinct projective irrep characters with multiplier $\alpha$ as per~\eqref{eq: burnside eigenvalue equation}.
    \end{enumerate}
\end{theorem}

\begin{theorem}[Dixon's algorithm for character tables of projective irreps]
\label{thm: dixon character table}
    Let $\alpha$ be a unitary multiplier on $G$. 
    Let $e$ be the exponent of $\alpha$ (\Cref{def: exponent of alpha}) and let $\zeta$ be a  primitive $e$th root of unity in $\C$.
    Let $p>2\sqrt{|G|}$ be a prime such that $e$ divides $p-1$ and let $z\in \Zp$ be such that $z^e=1$ and $z^n\neq1$ for all $0<n<e$.\footnote{Such $p$ and $z$ exist due to Dirichlet's theorem and the existence of a primitive element of multiplicative order $p-1$ in $\Zp^\times$.}
    Let 
    \begin{equation}
        \theta:\Zzeta\to\Zp,\quad f(\zeta)\mapsto f(z)\bmod p\ \text{for all polynomials $f(x)\in \mathbb{Z}[x]$.}
        \label{eq: embedding in Zp}
    \end{equation}
    Then
    \begin{enumerate}
        \item $\theta$ is a well-defined ring homomorphism from $\Zzeta$ to $\Zp$.
        \item Let $(\pi,V,\alpha)$ be a projective representation of $G$. Then, for all $g\in G$, the eigenvalues of $\pi(g)$ are $e$th roots of unity. In particular, the characters $\chi(g)$ of $\pi$ belong to $\Zzeta$.
        \item Using the notation of \Cref{thm: burnside}, define the $m\times m$ matrices $\tilde M_A$ for each $\alpha$-regular class $A$ as
        \begin{equation}
            (\tilde M_A)_{BC}= \sum_{\substack{a\in A \\ab_0\in C}}\frac{\theta(\beta(c_0,ab_0))}{\theta(\beta(a_0,a))}\theta(\alpha(a,b_0)).
            \label{eq: dixon M matrix}
        \end{equation}
        Let $\chi$ be the character of an irreducible projective representation of degree $d$ of $G$ with multiplier $\alpha$. Then the eigenvalue equation
        \begin{equation}
            \tilde M_A \left(\begin{array}{c}
                  \theta\big(\chi\big(c^{(1)}_0\big)\big)\\
                 \vdots\\
                  \theta\big(\chi\big(c^{(m)}_0\big)\big)\\
            \end{array}\right) 
            = 
            |A|\theta(\chi(a_0)) d^{-1} \left(\begin{array}{c}
                  \theta\big(\chi\big(c^{(1)}_0\big)\big)\\
                 \vdots\\
                  \theta\big(\chi\big(c^{(m)}_0\big)\big)\\
            \end{array}\right) 
            \label{eq: dixon eigenvalue equation}
        \end{equation}
        holds in $\Zp$ for each $\alpha$-regular class $A$.
        In other words, the vector on the right hand side is a joint eigenvector of all the matrices $\tilde M_A$.
        \item The matrices $\tilde M_A$ are jointly diagonalisable in $\Zp$. Their $m$  joint eigenvectors are uniquely defined up to a scalar factor and correspond to the distinct projective irrep characters with multiplier $\alpha$ as per~\eqref{eq: dixon eigenvalue equation}.
        \item Let $\chi$ be the character of a projective representation $(\pi,V,\alpha)$. Let $m_k(g)\ge 0$ be the multiplicity of the eigenvalue $\zeta^k$ in $\pi(g)$ for all $g\in G$ and $0\le k< e$.\footnote{By (1), all eigenvalues of all the $\pi(g)$ are $\zeta^k$ for some $0\le k<e$.} Then
        \begin{enumerate}
            \item $\chi(g^n) = \big[\prod_{i=1}^{n-1} \alpha(g^i,g)\big]\big[\sum_{k=0}^e m_k(g) \zeta^{kn}\big]$ for all $g\in G$ and $n\ge 0$.
            \item For all $g\in G$ and $0\le k< e$,
            \begin{equation}
                m_k(g) = e^{-1}\sum_{n=0}^{e-1} z^{-nk}  \theta(\chi(g^n)) \prod_{i=1}^{n-1}\theta(\alpha(g^i,g))^{-1} \pmod p.
                \label{eq: dixon fourier}
            \end{equation}
            In particular, if $\pi$ is irreducible, $m_k$ is the smallest such nonnegative integer.
        \end{enumerate}
    \end{enumerate}
\end{theorem}

\begin{theorem}[Dixon's algorithm for decomposing projective representations into irreps]
    \label{thm: dixon irrep matrix}
    Let $(\pi, V, \alpha)$ be a projective representation of $G$. For any linear map $f:V\to V$, define
    \begin{equation*}
        \langle f\rangle_\pi := \frac1{|G|}\sum_{g\in G} \pi(g)^{-1} f \pi(g).
    \end{equation*}
    \begin{enumerate}
        \item The eigenspaces of $\langle f\rangle_\pi$ are projective representations of $G$ with multiplier $\alpha$.
        \item Let $\{f_i\}_{i=1}^{(\dim V)^2}$ be a basis of the vector space of linear maps $V\to V$. Then $\pi$ is irreducible if and only if $\langle f_i\rangle_\pi$ is a multiple of $\id_V$ for all $1\le i\le (\dim V)^2$.
    \end{enumerate}
\end{theorem}

The rest of the paper is organised as follows. 
In \Cref{sec: preliminaries}, we define Schur multipliers, projective representations and their characters, and $\alpha$-class functions, and review those results of  Ref.~\cite{Cheng2015CharacterTheory} that we make use of in the rest of the paper.

In \Cref{sec: burnside}, we prove \Cref{thm: burnside}, which gives an immediate prescription for computing the character table of projective irreps. This is especially convenient in floating-point arithmetic, as it bypasses constructing the representation group $G^*$ based on approximate values of the multiplier.

In \Cref{sec: dixon character table}, we review the method used in Ref.~\cite{Dixon1967} to transpose Burnside's algorithm to a finite field and use it to prove \Cref{thm: dixon character table}. We also discuss some technical points of turning \Cref{thm: dixon character table} into a practical algorithm to compute the character table, see in particular \Cref{alg: normalisation mod p}. Nevertheless, the goal of this paper is to illustrate how the character theory of Ref.~\cite{Cheng2015CharacterTheory} can be used to bypass the representation group $G^*$ in algorithms, rather than to provide a highly optimised implementation. Accordingly, we do not discuss the technical improvements to Dixon's original algorithm in Ref.~\cite{Schneider1990}, or newer algorithms such as~\cite{Unger2006}.

Finally, in \Cref{sec: dixon irrep matrix}, we prove \Cref{thm: dixon irrep matrix} and discuss strategies to use it to compute matrices for all projective irreps with a given multiplier.

\subsection*{Acknowledgements}

I thank David Craven, Olivér Janzer, and Stacey Law for their careful reading of and feedback on the manuscript.

\section{Review of the character theory of projective representations}
\label{sec: preliminaries}

\subsection{(Schur) multipliers} 
\begin{definition}
    \label{def: multiplier}
    A function $\alpha:G\times G\to \C^\times$ is called a \emph{(Schur) multiplier} (or a \textit{factor set, factor system,} or \emph{2-cocycle}) on $G$ if
    \begin{enumerate}
        \item $\alpha(x,y)\alpha(xy,z) = \alpha(x,yz)\alpha(y,z)$ for all $x,y,z\in G$
        \item $\alpha(x,1)=\alpha(1,x)=1$ for all $x\in G$.
    \end{enumerate}
    We say that a multiplier is \emph{unitary} if $\alpha(x,y)$ is a root of unity for all $x,y\in G$, i.e., there exists a positive integer $N$ such that $\alpha(x,y)^N=1$ for all $x,y\in G$. We call the smallest such $N$ the \textit{order of the unitary multiplier $\alpha$.}
\end{definition}

\begin{lemma}[\cite{Cheng2015CharacterTheory}, \S1]
    \begin{enumerate}
        \item Multipliers on $G$ form an abelian group with pointwise multiplication as the group operation. We denote this group by $Z^2(G,\C)$.
        \item Let $\mu:G\to\C^\times$ be an arbitrary function with $\mu(1)=1$. Then $\alpha(x,y) = \mu(xy)/[\mu(x)\mu(y)]$ is a multiplier on $G$.
        \item The set of all multipliers of the above form is a normal subgroup of $Z^2(G,\C)$. We label this group by $B^2(G,\C)$ and the quotient group (known as the \emph{2-cohomology group}) by $H^2(G,\C)$. We denote the image of a multiplier $\alpha$ in $H^2(G,\C)$ by $[\alpha]$.
    \end{enumerate}
\end{lemma}

\begin{lemma}[\cite{Cheng2015CharacterTheory}, Lemma 3.1]
    \label{thm: unitary multipliers}
    Let $\alpha$ be a multiplier on $G$. Then there exists a unitary multiplier $\alpha'$ such that $[\alpha]=[\alpha']$.
\end{lemma}

\subsection{Projective representations}

\begin{definition}
    Let $V$ be a finite-dimensional vector space over $\C$. A \emph{projective representation} of $G$ over $V$ with multiplier $\alpha$ is a function $\pi:G\to\mathrm{GL}(V)$ such that $\pi(x)\pi(y)=\alpha(x,y)\pi(xy)$ for all $x,y\in G$. We denote this projective representation as $(\pi,V,\alpha)$. We call $\dim V$ the \emph{degree} of the projective representation and denote it by $\dim \pi$.

    A projective representation is called \emph{unitary} if the associated multiplier $\alpha$ is unitary.
\end{definition}

\begin{remark}
    Linear representations are projective representations with the associated multiplier $\alpha(x,y)=1$ for all $x,y\in G$. Accordingly, all the following results hold (sometimes trivially) for linear representations too.
\end{remark}

The last item of this definition deserves some elaboration:
\begin{lemma}
    \label{thm: unitary reps are unitary}
    Let $(\pi,V,\alpha)$ be a projective representation of $G$ with a unitary multiplier $\alpha$. Then $\pi(g)$ is unitary for all $g\in G$.
\end{lemma}
\begin{proof}
    Let $g\in G$ and let $n$ be the order of $g$ in $G$. Then $\pi(g)^n = A\id_V$, where $A = \prod_{i=1}^{n-1} \alpha(g^i, g)$. Since $\alpha$ is unitary, there is an integer $N$ such that $\alpha(x,y)^N=1$ for all $x,y\in G$, and so $A^N=1$. Therefore $\pi(g)^{nN}=\id_V$, which can only be the case if $\pi(g)$ is unitary. 
\end{proof}
\begin{definition}
    \label{def: order of g in alpha}
    Let $\alpha$ be a unitary multiplier on $G$. Let $g\in G$ of order $n$. As in the proof above, let $A=\prod_{i=1}^{n-1} \alpha(g^i, g)$, and let $n^*$ be the smallest positive integer such that $A^{n^*}=1$. Then we call $nn^*$ the \emph{order of $g$ in $\alpha$} and denote it by $o_\alpha(g)$.
    
    As explained above, $\pi(g)^{o_\alpha(g)} = \id_V$ for all $g\in G$ for all projective representations $(\pi,V,\alpha)$.
\end{definition}

\begin{definition}
    Two projective representations $(\pi,V,\alpha)$ and $(\pi',V',\alpha')$ of $G$ are said to be \emph{equivalent} if there exists a linear isomorphism $\phi:V\to V'$ and a function $\mu:G\to\C^\times$ with $\mu(1)=1$ such that $\mu(g)\pi'(g)\phi = \phi\pi(g)$ for all $g\in G$.
\end{definition}
\begin{lemma}[\cite{Cheng2015CharacterTheory}, Lemma 2.9]
    Let $(\pi,V,\alpha)$ be a projective representation of $G$ and $\alpha'$ a multiplier such that $[\alpha]=[\alpha']$. Then there exists a projective representation $(\pi',V',\alpha')$ that is equivalent to $(\pi,V,\alpha)$. In particular, given \Cref{thm: unitary multipliers}, there exists a unitary projective representation equivalent to $(\pi,V,\alpha)$.
\end{lemma}

Therefore, we can obtain a full description of projective representations of finite groups by studying unitary projective representations only.

\begin{definition}
    \label{def: irrep}
    A projective representation $(\pi,V,\alpha)$ is called \textit{irreducible} if there are no nontrivial proper subspaces of $V$ that are invariant under $\pi(g)$ for all $g\in G$. We also use the contraction `projective irrep' for `irreducible projective representation'.
\end{definition}

\begin{proposition}[\cite{Cheng2015CharacterTheory}, Theorem 5.6]
    \label{thm: irrep dim divides order}
    The degrees of irreducible projective representations of $G$ divide $|G|$.
\end{proposition}

Irreps play a central role in the study of projective representations for much the same reasons as in the case of linear representations: All projective representations can be decomposed into irreps, and Schur's lemma~\cite[Theorem~2.12]{Cheng2015CharacterTheory} holds for them.

\begin{proposition}[Complete reducibility; \cite{Cheng2015CharacterTheory}, Theorem 2.3]
    Every projective representation of $G$ is a direct sum of irreducible projective representation with the same multiplier.
\end{proposition}

\begin{proposition}[Corollary of Schur's lemma; \cite{Cheng2015CharacterTheory}, Corollary 2.13]
    \label{thm: schur lemma}
    Let $(\pi,V,\alpha)$ be an irreducible projective representation of $G$. Let $f:V\to V$ be a linear map. Then,
    \begin{equation*}
        \langle f\rangle_\pi = \frac1{|G|}\sum_{g\in G} \pi(g)^{-1} f \pi(g) = \frac{\tr (f)}{\dim V} \id_V,
    \end{equation*}
    where $\id_V$  is the identity map on $V$.
\end{proposition}
\begin{corollary}
    \label{thm: schur shuffled}
    Since $\pi(g^{-1}) = \alpha(g,g^{-1})\pi(g)^{-1}$ for all $g\in G$, we also have
    \begin{equation*}
        \frac1{|G|}\sum_{g\in G} \pi(g) f \pi(g^{-1}) = \langle f\rangle_\pi = \frac{\tr (f)}{\dim V} \id_V,
    \end{equation*}
\end{corollary}

Duals and tensor products of projective representations are also projective representations, albeit with different multipliers \cite[\S2.5]{Cheng2015CharacterTheory}:
\begin{lemma}
    \label{thm: dual}
    Let $(\pi,V,\alpha)$ be a projective representation of $G$. Let $V^* = \hom(V,\C)$ be the dual of $V$ and let
    \begin{equation*}
        \pi^*(g):V^*\to V^*, \pi^*(g) = \frac{^t\pi(g^{-1})}{\alpha(g,g^{-1})}
    \end{equation*}
    for all $g\in G$. Then $\pi^*$ is a projective representation of $G$ over $V^*$, with multiplier $\alpha^{-1}$.
\end{lemma}
\begin{lemma}
    \label{thm: tensor product}
    Let $(\pi,V,\alpha)$ and $(\pi',V',\alpha')$ be projective representations of $G$.    Let $(\pi\otimes\pi')(g) = \pi(g)\otimes \pi'(g)$ for all $g\in G$. Then $\pi\otimes\pi'$ is a projective representation of $G$ over $V\otimes V'$ with multiplier $\alpha\alpha'$.
\end{lemma}

Regular projective representations are also defined much the same way as in the linear case:
\begin{lemma}
    Let $V$ be a $|G|$-dimensional vector space, with a basis $\{e_g:g\in G\}$. Let $\alpha$ be a multiplier on $G$.
    \begin{enumerate}
        \item Let $L(h)$ be the linear map defined by $L(h)\cdot e_g = \alpha(h,g)e_{hg}$ for all $g,h$.
        \item Let $R(h)$ be the linear map defined by $R(h)\cdot e_g = \alpha(h,g^{-1})e_{gh^{-1}}$ for all $g,h$.
    \end{enumerate}
    Then $(L,V,\alpha)$ and $(R,V,\alpha)$ are projective representations of $G$, called respectively \emph{left} and \emph{right $\alpha$-regular projective representations}.
\end{lemma}
\begin{proof}
    For all $x,y,g\in G$,
    \begin{enumerate}
        \item $L(x)\cdot L(y)\cdot e_g = \alpha(y,g)L(x)\cdot e_{yg} = \alpha(x,yg)\alpha(y,g) e_{xyg} = \alpha(xy,g)\alpha(x,y) e_{xyg} = \alpha(x,y)L(xy)\cdot e_g$;
        \item $R(x)\cdot R(y)\cdot e_g = \alpha(y,g^{-1})L(x)\cdot e_{gy^{-1}} = \alpha(x,yg^{-1})\alpha(y,g^{-1}) e_{gy^{-1}x^{-1}} = \alpha(xy,g^{-1})\alpha(x,y) e_{g(xy)^{-1}} = \alpha(x,y)R(xy)\cdot e_g$,
    \end{enumerate}
    as required for a projective representation.
\end{proof}

\subsection{Characters of projective representations and $\alpha$-class functions}

\begin{definition}
    The \textit{character} $\chi:G\to\C$ of a projective representation $(\pi,V,\alpha)$ of $G$ is given by $\chi(g) = \tr (\pi(g))$ for all $g\in G$.
\end{definition}

Similar to the case of linear representations, characters specify projective representations essentially uniquely:
\begin{lemma}[\cite{Cheng2015CharacterTheory}, Corollary 3.8]
    Two unitary projective representations with the same multiplier and the same character are equivalent.
\end{lemma}

Many basic properties of the characters of linear representations carry over with slight changes:
\begin{lemma}[\cite{Cheng2015CharacterTheory}, Lemma 3.4]
    \label{thm: character basics}
    Let $\chi$ be the character of a unitary projective representation $(\pi,V,\alpha)$ of $G$. Then
    \begin{enumerate}
        \item $\chi(1) = \dim \pi$
        \item For all $g\in G$, $\chi(g^{-1}) = \alpha(g,g^{-1}) \overline{\chi(g)}$.
        \item For all $g,h\in G$,
        \begin{equation}
            \chi(hgh^{-1}) = \frac{\alpha(h,h^{-1})}{\alpha(h,gh^{-1})\alpha(g,h^{-1})}\chi(g) = \frac{\alpha(h,h^{-1})}{\alpha(h,g)\alpha(hg,h^{-1})}\chi(g).
            \label{eq: character conjugation}
        \end{equation}
    \end{enumerate}
\end{lemma}
\begin{proof}
    \begin{enumerate}
        \item Take the trace of $\pi(1)=\id_V$.
        \item Let $g\in G$. By \Cref{thm: unitary reps are unitary}, $\pi(g)$ is unitary, so we have $\pi(g^{-1}) = \alpha(g,g^{-1})\pi(g)^{-1} = \alpha(g,g^{-1})\pi(g)^*$. The claim follows by taking the trace of both sides.
        \item Let $g,h\in G$. Then we have 
        \begin{equation}
            \pi(hgh^{-1}) = \frac{\alpha(h,h^{-1})}{\alpha(h,gh^{-1})\alpha(g,h^{-1})} \pi(h)\pi(g)\pi(h)^{-1} = \frac{\alpha(h,h^{-1})}{\alpha(h,g)\alpha(hg,h^{-1})}\pi(h)\pi(g)\pi(h)^{-1};
            \label{eq: rep conjugation}
        \end{equation}
        the claim follows by taking the trace and noting that $\pi(h)\pi(g)\pi(h)^{-1}$ and $\pi(g)$ are unitary similar.
    \end{enumerate}
\end{proof}

In line with (3), we need to generalise the notion of class functions:

\begin{definition}[\cite{Cheng2015CharacterTheory}, Definition 3.13]
    \label{def: class function}
    Let $\alpha$ be a multiplier on $G$. A function $f:G\to \C$ is called an \textit{$\alpha$-class function} if for all $g,h\in G$,
    \begin{equation*}
        f(hgh^{-1}) = \frac{\alpha(h,h^{-1})}{\alpha(h,gh^{-1})\alpha(g,h^{-1})}f(g) = \frac{\alpha(h,h^{-1})}{\alpha(h,g)\alpha(hg,h^{-1})}f(g).
    \end{equation*}
    In particular, the characters of unitary projective representations $(\pi,V,\alpha)$ are $\alpha$-class functions.
\end{definition}

\begin{definition}
    Let $\alpha$ be a multiplier on $G$. An element $g\in G$ is said to be \textit{$\alpha$-regular} if 
    \begin{equation*}
        \frac{\alpha(h,h^{-1})}{\alpha(h,gh^{-1})\alpha(g,h^{-1})} = \frac{\alpha(h,h^{-1})}{\alpha(h,g)\alpha(hg,h^{-1})} = 1
    \end{equation*} 
    for all $h\in G$ such that $gh=hg$. 
\end{definition}

From these definitions, it is clear that an $\alpha$-class function can only be nonzero for $\alpha$-regular elements of $G$. In particular, the character of any unitary projective representation $(\pi,V,\alpha)$ vanishes for all non-$\alpha$-regular elements of $G$.

\begin{lemma}[\cite{Cheng2015CharacterTheory}, Lemma 3.16(2)]
    \label{thm: regular classes}
    If $g\in G$ is $\alpha$-regular, so are the conjugates of $g$. The resulting conjugacy classes are also said to be \emph{$\alpha$-regular.}
\end{lemma}

\begin{proposition}[\cite{Cheng2015CharacterTheory}, Theorem 3.15]
    \label{thm: orthonormality}
    Let $\alpha$ be a unitary multiplier on $G$ and let $\mathbb{H}_\alpha$ be the space of $\alpha$-class functions on $G$, equipped with the inner product
    \begin{equation}
        (\phi,\psi) = \frac1{|G|}\sum_{g\in G} \phi(g)\overline{\psi(g)}.
    \end{equation}
    Then the characters of inequivalent projective irreps with multiplier $\alpha$ form an orthonormal basis of $\mathbb{H}_\alpha$.
\end{proposition}

\begin{corollary}
    \label{thm: characters equals classes}
    Let $\alpha$ be a unitary multiplier on $G$. Then $\dim\mathbb{H}_\alpha$ is equal to both the number of inequivalent projective irreps $(\pi,V,\alpha)$ of $G$ and the number of $\alpha$-regular conjugacy classes of $G$.
\end{corollary}

\begin{lemma}[\cite{Cheng2015CharacterTheory}, Corollary 3.11]
    \label{thm: regular decomposition}
    Let $\alpha$ be a multiplier on $G$.
    \begin{enumerate}
        \item Every projective irrep of $G$ with multiplier $\alpha$ appears in the left and right $\alpha$-regular projective representations of $G$ with multiplicity equal to its degree.
        \item The degrees $d_i$ of inequivalent projective irreps with multiplier $\alpha$ satisfy $\sum_i d_i^2 = |G|$.
    \end{enumerate}
\end{lemma}

\section{Burnside's algorithm for projective irrep characters}
\label{sec: burnside}

We start this section by proving some lemmas in~\S\ref{sec: class factors} about Schur multipliers that need to hold for a consistent definition of projective representations and $\alpha$-class functions.
The strategy we then take to prove \Cref{thm: burnside} is rather similar to the case of linear representations~\cite{Burnside1911}.
First, using the results of \S\ref{sec: class factors}, we construct a linear combination of $\pi(g)$ in each given conjugacy class, $F(g)$, which is proportional to the identity (\Cref{thm: conjugacy class sum}). 
The matrices $M$ in~\eqref{eq: burnside M matrix} are then obtained by multiplying pairs of these.

\subsection{Consistency relations for $\alpha$-class functions}
\label{sec: class factors}

\begin{definition}
    Let $\alpha$ be a multiplier on $G$ and let $g,h\in G$.
    We shall call the expressions
    \begin{equation*}
        \beta_h(g) := \frac{\alpha(h,h^{-1})}{\alpha(h,gh^{-1})\alpha(g,h^{-1})} = \frac{\alpha(h,h^{-1})}{\alpha(h,g)\alpha(hg,h^{-1})}
    \end{equation*}
    \textit{class factors.}
\end{definition}

The fact that~\eqref{eq: rep conjugation} holds consistently across a conjugacy class puts constraints on these class factors:

\begin{lemma}
    \label{thm: product rule}
    Let $\alpha$ be a multiplier on $G$. Then for all $x,h,k\in G$, the corresponding class factors satisfy $\beta_h(x)\beta_k(hxh^{-1}) = \beta_{kh}(x)$.
\end{lemma}
\begin{proof}
    The claim follows from a straightforward if tedious calculation:
    \begin{align*}
        \beta_{kh}(x) &= \frac{\alpha(kh,h^{-1}k^{-1})}{\alpha(kh,xh^{-1}k^{-1})\alpha(x,h^{-1}k^{-1})}\\
        \alpha(kh,h^{-1}k^{-1}) &= 
        \frac{\alpha(k,k^{-1})\alpha(kh,h^{-1})}{\alpha(h^{-1},k^{-1})} = 
        \frac{\alpha(k,k^{-1})\alpha(h,h^{-1})}{\alpha(h^{-1},k^{-1})\alpha(k,h)}\\
        \alpha(kh,xh^{-1}k^{-1}) &= 
        \frac{\alpha(k,hxh^{-1}k^{-1})\alpha(h,xh^{-1}k^{-1})}{\alpha(k,h)} = 
        \frac{\alpha(k,hxh^{-1}k^{-1})\alpha(h,xh^{-1})\alpha(hxh^{-1},k^{-1})}{\alpha(k,h)\alpha(xh^{-1},k^{-1})}\\
        \alpha(x,h^{-1}k^{-1}) &= 
        \frac{\alpha(xh^{-1},k^{-1})\alpha(x,h^{-1})}{\alpha(h^{-1},k^{-1})} = 
        \frac{\alpha(h^{-1},hxh^{-1}k^{-1})\alpha(hxh^{-1},k^{-1})\alpha(x,h^{-1})}{\alpha(h^{-1},hxh^{-1})\alpha(h^{-1},k^{-1})}\\
        \therefore \beta_{kh}(x) &= \frac{\alpha(h,h^{-1})}{\alpha(h,xh^{-1})\alpha(x,h^{-1})} \frac{\alpha(k,k^{-1})}{\alpha(k,hxh^{-1}k^{-1})\alpha(hxh^{-1},k^{-1})} \frac{\alpha(xh^{-1},k^{-1})\alpha(h^{-1},hxh^{-1})}{\alpha(h^{-1},hxh^{-1}k^{-1})\alpha(hxh^{-1},k^{-1})}.
    \end{align*}
    The first two terms are $\beta_h(x)$ and $\beta_k(hxh^{-1})$, respectively.
    Substituting $x=h^{-1}, y=hxh^{-1}, z=k^{-1}$ into~\Cref{def: multiplier}(2), we find that the last term is 1.
\end{proof}

\begin{lemma}
    \label{thm: beta uniqueness}
    If $g\in G$ is $\alpha$-regular, the class factor $\beta_h(g)$ is only a function of the conjugate elements $g$ and $g'=hgh^{-1}$, but not $h$. In this case, we shall write $\beta(g,g') :=\beta_h(g)$.
\end{lemma}
\begin{proof}
    Let $h,k\in G$ such that $g' = hgh^{-1} = kgk^{-1}$. Then $h^{-1}kgk^{-1}h = g$, so $k=hc$, where $c\in C_G(g)$. \Cref{thm: product rule} implies $\beta_c(g) \beta_h(g) = \beta_k(g)$, and since $g$ is $\alpha$-regular, $\beta_h(g) = \beta_k(g)$.
\end{proof}

\begin{corollary}
    Let $g,g',g''\in G$ be in the same $\alpha$-regular conjugacy class. Then the class factors satisfy $\beta(g,g')\beta(g',g'') = \beta(g,g'')$.
\end{corollary}
\begin{remark}
    This must be the case in order for~\eqref{eq: character conjugation} to hold consistently for nonzero characters.
\end{remark}

\subsection{Proof of \Cref{thm: burnside}}

\begin{lemma}
    \label{thm: conjugacy class sum}
    Let $g\in G$. Let $(\pi, V, \alpha)$ be projective irrep of $G$  with characters $\chi$, and let $\beta$ be the class factors corresponding to $\alpha$. Then
    \begin{equation}
        F(g):= \frac1{|G|}\sum_{h\in G} \frac{\pi(hgh^{-1})}{\beta_h(g)} = \frac{\chi(g)}{\dim V}\id_V. 
        \label{eq: conjugacy class sum}
    \end{equation}
\end{lemma}
\begin{proof}
    By~\eqref{eq: rep conjugation}, 
    \begin{equation}
        F(g) = \frac1{|G|} \sum_{h\in G} \pi(h)\pi(g)\pi(h)^{-1} = \langle \pi(g)\rangle_\pi = \frac{\chi(g)}{\dim V}\id_V;
        \label{eq: F matrix product form}
    \end{equation}
    the last equality holds due to \Cref{thm: schur shuffled}.
\end{proof}
\begin{remark}
    If $g$ and $g'$ are in the same $\alpha$-regular conjugacy class $C$, the above can be written as
    \begin{align*}
        F(g) &= \frac1{|C|} \sum_{g'\in C} \frac{\pi(g')}{\beta(g,g')} = \frac{\chi(g)}d\id.
    \end{align*}
    If $g$ is not $\alpha$-regular, $F(g) = 0$.
\end{remark}

\begin{proof}[Proof of \Cref{thm: burnside}]
    By \Cref{thm: conjugacy class sum},
    \begin{align}
        \frac1{d^2}|A||B|\chi(a_0)\chi(b_0)\id = |A||B| F(a_0)F(b_0) = \sum_{a\in A}\sum_{b\in B} \frac{\pi(a)}{\beta(a_0,a)} \frac{\pi(b)}{\beta(b_0,b)}.
        \label{eq: class homothety product}
    \end{align}
    We would like to split this sum into terms of the form~\eqref{eq: conjugacy class sum}, so we can express it as a linear combination of the representative $\chi(c_0)$ of the several conjugacy classes. To do so, it is convenient to introduce an additional sum over $h\in G$ and reparametrise the sums over $a\in A,b\in B$ as $a\mapsto hah^{-1}, b\mapsto hbh^{-1}$:
    \begin{align*}
        |A||B| F(a_0)F(b_0) &= 
        \frac1{|G|}\sum_{a\in A}\sum_{b\in B} \sum_{h\in G} \frac{\pi(hah^{-1})}{\beta(a_0,a)\beta_h(a)} \frac{\pi(hbh^{-1})}{\beta(b_0,b)\beta_h(b)} \qquad\text{(\Cref{thm: product rule})}\\
        &= \frac1{|G|}\sum_{a\in A}\sum_{b\in B} \frac1{\beta(a_0,a)\beta(b_0,b)} \sum_{h\in G} \pi(h)\pi(a)\pi(h)^{-1} \times \pi(h)\pi(b)\pi(h)^{-1} \qquad\eqref{eq: rep conjugation}\\
        &= \frac1{|G|}\sum_{a\in A}\sum_{b\in B} \frac{\alpha(a,b)}{\beta(a_0,a)\beta(b_0,b)} \sum_{h\in G} \pi(h)\pi(ab)\pi(h)^{-1} \\
        &= \frac{\id}d\sum_{a\in A}\sum_{b\in B} \frac{\alpha(a,b)\chi(ab)}{\beta(a_0,a)\beta(b_0,b)}\qquad\eqref{eq: F matrix product form} \\
        &= \frac{\id}d \sum_C \chi(c_0)\sum_{\substack{a\in A \\ b\in B\\ab\in C}}\frac{\beta(c_0,ab)}{\beta(a_0,a)\beta(b_0,b)}\alpha(a,b).\qquad\text{(\Cref{thm: character basics}(3))}
    \end{align*}
    Note that $\chi(ab)=0$ if $ab$ is not $\alpha$-regular, so the sum over $C$ can be taken to run over only $\alpha$-regular classes. Comparing the two ends of the chain of equations and inserting~\eqref{eq: burnside M matrix}, we get
    \begin{align}
        \sum_C \sqrt{|B||C|} (M_A)_{BC} \chi(c_0) &= \frac{|A|\chi(a_0)}d |B|\chi(b_0) \nonumber\\
        \sum_C (M_A)_{BC} \frac{\sqrt{|C|}\chi(c_0)}{\sqrt{|G|}} &= \frac{|A|\chi(a_0)}d \frac{\sqrt{|B|}\chi(b_0)}{\sqrt{|G|}},
        \label{eq: eigenvalue equation components}
    \end{align}
    which is precisely~\eqref{eq: burnside eigenvalue equation} written out in components. This concludes the proof of part (1).

    By \Cref{thm: characters equals classes}, there are $m$ inequivalent projective irreps with characters $\chi_i$ $(1\le i \le m)$, which give rise to $m$ linearly independent eigenvectors $v_i$ of $M_A$ by the process above. Since $m\times m$ matrices can have up to $m$ linearly independent eigenvectors, $M_A$ has no eigenvectors linearly independent of these. 
    The eigenvalues associated with eigenvector $v_i$ are $\big(|C^{(1)}|\chi_i\big(c^{(1)}_0\big)/d_i, \dots, \big(|C^{(m)}|\chi_i\big(c^{(m)}_0\big)/d_i\big)$: by \Cref{thm: orthonormality}, these vectors are linearly independent, so no linear combination of more than one $v_i$ can be a joint eigenvector of all the $M_A$.
    
    That is, all the $M_A$ share all their eigenvectors $v_i$ ($1\le i\le m$), which are of the form given in~\eqref{eq: burnside eigenvalue equation}. This implies that all the $M_A$ commute. Finally, by \Cref{thm: orthonormality}, the eigenvectors $v_i$ are orthonormal under the usual inner product on $\C^m$, therefore, the $M_A$ are normal.
\end{proof}
\begin{remarks}
    \begin{enumerate}
        \item In the special case $\alpha(g,h)=1$ for all $g,h\in G$, we recover the standard Burnside algorithm.
        \item Since the eigenvectors in Eq.~\eqref{eq: burnside eigenvalue equation} are orthonormal under the usual inner product, the columns of the character table are orthogonal too:
        \begin{equation*}
            \sum_{i=1}^m \chi_i\big(c_0^{m}\big) \overline{\chi_j\big(c_0^{m}\big)} = \frac{|G|}{|C^{(i)}|}\delta_{ij}.
        \end{equation*}
    \end{enumerate}
\end{remarks}

\section{Dixon's algorithm for computing exact characters}
\label{sec: dixon character table}

\subsection{Brief review of Dixon's character table algorithm}

\Cref{thm: burnside} gives a straightforward algorithm to compute projective irrep characters, which is quite well suited for floating-point numerical computation.
However, solving the eigenvalue problem analytically becomes a challenge already for rather small groups, limiting its usefulness for theoretical applications.

Dixon's character table algorithm~\cite{Dixon1967} sidesteps the issue by transposing the eigenvalue problem~\eqref{eq: burnside eigenvalue equation} on $\C$ to an appropriately chosen finite field. This is possible because every character of $G$ is the sum of $e$th roots of unity, where $e$ is the exponent of the group, i.e., the least common multiple of the orders of all $g\in G$: In a finite field with a primitive $e$th root of unity $z$, this $z$ can play the role of the primitive $e$th root of unity in $\C$.
\begin{lemma}
    \label{thm: Zzeta to Zp}
    Let $e$ be a positive integer and let $\zeta$ be a primitive $e$th root of unity. Let $p$ be a prime such that $e$ divides $p-1$ and let $z$ be a primitive $e$th root of unity in the field $\Zp$. Then
    \begin{equation*}
        \theta:\Zzeta\to\Zp,\quad f(\zeta)\mapsto f(z)\bmod p\ \text{for all polynomials $f(x)\in \mathbb{Z}[x]$}
    \end{equation*}
    is a well-defined ring homomorphism.
\end{lemma}
\begin{proof}
    We can view $\Zzeta$ as the quotient ring $\mathbb{Z}[x]/\Phi_e(x)$, where $\Phi_e(x)$ is the minimal polynomial of $\zeta$, the $e$th cyclotomic polynomial. $\theta$ is trivially a ring homomorphism from $\mathbb{Z}[x]$ to $\Zp$. What is left to prove is that it is a well-defined function $\Zzeta\to\Zp$, i.e., that $f(\zeta)=g(\zeta) \implies f(z)=g(z)\pmod p$ for all one-variable polynomials $f,g$ with integer coefficients. By the definition of the minimal polynomial,  $f(\zeta)=g(\zeta)$ implies that $\Phi_e$ divides $f-g$. Since cyclotomic polynomials are monic, this quotient has integer coefficients, i.e., $f = g + h\Phi_e$ for some polynomial $h\in\mathbb{Z}[x]$. Then, since $z$ is a primitive $e$th root of unity in $\Zp$, $\Phi_e(z)=0\pmod p$, so $f(z)=g(z)\pmod p$, as required.
\end{proof}
\begin{proof}[Remark]
    \Cref{thm: dixon character table}(1) is a special case of this result.
\end{proof}

This allows us to rewrite the eigenvalue problem in Burnside's algorithm in $\Zp$, where it can be solved using exact integer arithmetic. We can then transpose the result back to $\C$ using a discrete Fourier transform:
\begin{lemma}
    \label{thm: fourier}
    We use the notation of \Cref{thm: Zzeta to Zp}. Let $m_k$ be an integer for all $0\le k<e$ and let $\chi_j = \sum_{k=0}^{e-1} m_k z^{jk} \pmod p$ for all $0\le j < e$. Then
    $m_k = e^{-1}\sum_{j=0}^{e-1} \chi_j z^{-jk}\pmod p$.
\end{lemma}
\begin{proof}
    By definition, $\sum_{j=0}^{e-1} \chi_j z^{-jk} = \sum_{j=0}^{e-1} z^{-jk} \sum_{\kappa=0}^{e-1} m_\kappa z^{j\kappa} = \sum_{\kappa=0}^{e-1}m_\kappa\sum_{j=0}^{e-1}  z^{j(\kappa-k)}$. 
    
    Let $S(z^k) = \sum_{j=0}^{e-1} z^{jk}\pmod p$ for all $0\le k < e$. Clearly, $S(1)=e$. Otherwise, $z^k S(z^k) = \sum_{j=1}^{e} z^{jk} = S(z^k)\pmod p$ since $z^e=1\pmod p$ by definition, which implies $S(z^k)=0\pmod p$ for all $z^k\neq 1$. 
    
    Therefore, $\sum_{j=0}^{e-1} \chi_j z^{-jk} = em_k \pmod p$. The claim follows by multiplying both sides by $e^{-1}\pmod p$; by construction, such an inverse exists.
\end{proof}

\subsection{Generalisation to projective representations}

We can use the same idea to compute the character table of projective representations with unitary multipliers, since the eigenvalues of their matrices are also roots of unity.

\begin{definition}
    \label{def: exponent of alpha}
    Let $\alpha$ be a unitary multiplier on $G$. 
    The \emph{exponent} of $\alpha$ is the least common multiple 
    of the order of $\alpha$ (\Cref{def: multiplier}) and 
    of the orders of all elements of $G$ in $\alpha$ (\Cref{def: order of g in alpha}).
\end{definition}

\begin{proof}[Proof of \Cref{thm: dixon character table}(2)]
    By \Cref{def: order of g in alpha}, $\pi(g)^{o_\alpha(g)} = \id_V$ for all $g\in G$. By \Cref{def: exponent of alpha} then, $o_\alpha(g)$ divides $e$ for all $g\in G$, so $\pi(g)^e = \id_V$, which implies that all eigenvalues of $\pi(g)$ are $e$th roots of unity. These clearly belong to $\Zzeta$. Since the character $\chi(g)$ is the sum of these eigenvalues, it too belongs to $\Zzeta$.
\end{proof}

Before proving the rest of \Cref{thm: dixon character table}, we need to eliminate the square roots and fractions from the statement of \Cref{thm: burnside}:
\begin{corollary}
    \label{thm: burnside modified}
    We use the notation of \Cref{thm: burnside}. For each $\alpha$-regular class $A$, define the $m\times m$ matrix
    \begin{equation}
        (M'_A)_{BC}= \sum_{\substack{a\in A \\ab_0\in C}}\frac{\beta(c_0,ab_0)}{\beta(a_0,a)}\alpha(a,b_0).
        \label{eq: modified M matrix}
    \end{equation}
    \begin{enumerate}
        \item Let $\chi$ be the character of an irreducible projective representation of $G$ with multiplier $\alpha$. Then
        \begin{equation}
            |G|M_A' \left(\begin{array}{c}
                  \chi\big(c^{(1)}_0\big)\\
                 \vdots\\
                  \chi\big(c^{(m)}_0\big)\\
            \end{array}\right) 
            = 
            \frac{|G|}d|A|\chi(a_0) \left(\begin{array}{c}
                  \chi\big(c^{(1)}_0\big)\\
                 \vdots\\
                  \chi\big(c^{(m)}_0\big)\\
            \end{array}\right) 
            \label{eq: modified eigenvalue equation}
        \end{equation}
        \item All matrices $M_A$ commute. Their joint eigenvectors are uniquely defined up to a scalar factor and are all of the form~\eqref{eq: modified eigenvalue equation} for some projective irrep character $\chi$.
    \end{enumerate}
\end{corollary}
\begin{proof}
    One could show by brute-force calculation that $(M'_A)_{BC} = (M_A)_{BC}\sqrt{|C|/|B|}$, whence claim~(1) follows by substituting into~\eqref{eq: eigenvalue equation components} and rearranging.

    However, it is more instructive to alter how we evaluate~\eqref{eq: class homothety product}.  We replace the sum over $b\in B$ with one over $h\in G$ such that $b = hb_0h^{-1}$, which results in counting each term $|G|/|B|$ times. As before, we also reparametrise the sum over $a\in A$  as $a\mapsto hah^{-1}$:
    \begin{align}
        |A||B| F(a_0)F(b_0) 
        &= \frac{|B|}{|G|}\sum_{a\in A} \sum_{h\in G} \frac{\pi(hah^{-1})}{\beta(a_0,a)\beta_h(a)} \frac{\pi(hb_0h^{-1})}{\beta_h(b)} \qquad\text{(\Cref{thm: product rule})}\nonumber\\
        &= \frac{|B|}{|G|}\sum_{a\in A} \frac{\alpha(a,b_0)}{\beta(a_0,a)} \sum_{h\in G} \pi(h)\pi(ab_0)\pi(h)^{-1}\qquad\eqref{eq: rep conjugation} \nonumber\\
        &= \frac{|B|}d\id \sum_{a\in A} \frac{\alpha(a,b_0)\chi(ab_0)}{\beta(a_0,a)}\qquad\eqref{eq: F matrix product form} \nonumber\\
        &= \frac{|B|}d\id \sum_C \chi(c_0)\sum_{\substack{a\in A \\ab_0\in C}}\frac{\beta(c_0,ab_0)}{\beta(a_0,a)}\alpha(a,b_0).\qquad\text{(\Cref{thm: character basics}(3))} \nonumber\\
        \frac{|G|}d|A|\chi(a_0) \chi(b_0) &= |G| \sum_C \chi(c_0) (M'_A)_{BC},
        \label{eq: modified eigenvalue equation components}
    \end{align}
    as required. (We multiply both sides of~\eqref{eq: modified eigenvalue equation} to replace $1/d$ with $|G|/d$, which by \Cref{thm: irrep dim divides order} is an integer.)

    Claim~(2) follows by the same argument as \Cref{thm: burnside}(2).
\end{proof}

\begin{proof}[Proof of \Cref{thm: dixon character table}(3)]
    By \Cref{def: exponent of alpha}, $\alpha(x,y)^e=1$ for all $x,y\in G$, and so $\beta_h(g)^e=1$ for all $g,h\in G$. It follows that for all $A,B,C$, $(M'_A)_{BC}$ is a sum of $e$th roots of unity and so belongs to $\Zzeta$. Together with claim~(1), this means that every term in~\eqref{eq: modified eigenvalue equation components} belongs to $\Zzeta$. Applying $\theta$ to both sides and using the fact that it is a ring homomorphism, we get
    \begin{equation*}
        \frac{|G|}d|A|\theta(\chi(a_0)) \theta(\chi(b_0)) = |G| \sum_C \theta(\chi(c_0)) (\tilde M_A)_{BC}.
    \end{equation*}
    The claim follows by multiplying both sides by $|G|^{-1}\pmod p$. (By Cauchy's theorem, all prime factors of $|G|$ divide $e$, so $p$ cannot divide $|G|$. Therefore, such an inverse exists. Also, by \Cref{thm: irrep dim divides order}, $d$ divides $|G|$, so $d^{-1}\pmod p$ exists too.)
\end{proof}

\begin{proof}[Proof of \Cref{thm: dixon character table}(4)] 
    We first need to prove that the vectors $\theta(\chi)$ for distinct projective irrep characters $\chi$ are linearly independent in $\Zp$. From \Cref{thm: orthonormality}, we have $\sum_{i=1}^m |C^{(i)}|\chi_{j}\big(c^{(i)}_0\big)\overline{\chi_{k}}\big(c^{(i)}_0\big) = |G|\delta_{jk}$ for all $1\le j,k\le m$. Applying $\theta$ to both sides, we get 
    \begin{equation}
        \sum_{i=1}^m |C^{(i)}|\theta\big(\chi_{j}\big(c^{(i)}_0\big)\big)\theta\big(\overline{\chi_{k}}\big(c^{(i)}_0\big)\big) = |G|\delta_{jk} \pmod p.
        \label{eq: orthonormality mod p}
    \end{equation}
    Let $a_j$ ($1\le j\le m$) be integers such that $\sum_{j=1}^m a_j \theta(\chi_j(g)) = 0\pmod p$ for all class representatives $c_0$. Then
    \begin{align*}
        0 
        = \sum_{i=1}^m |C^{(i)}|\theta\big(\overline{\chi_{k}}\big(c^{(i)}_0\big)\big) \sum_{j=1}^m a_j \theta\big(\chi_j\big(c^{(i)}_0\big)\big) 
        = \sum_{j=1}^m a_j \sum_{i=1}^m |C^{(i)}|\theta\big(\chi_j\big(c^{(i)}_0\big)\big)\theta\big(\overline{\chi_{k}}\big(c^{(i)}_0\big)\big) = |G|a_j. \pmod p
    \end{align*}
    for all $1\le j\le m$. Since $p$ does not divide $|G|$, this implies that $p|a_j$ for all $1\le j\le m$, i.e., that there is no nontrivial linear dependence between the $\theta(\chi)$ in $\Zp$.

    Given this, the claim follows by the same argument as \Cref{thm: burnside}(2). 
    By \Cref{thm: characters equals classes}, there are $m$ inequivalent projective irreps with characters $\chi_i$ $(1\le i \le m)$, which give rise to $m$ linearly independent eigenvectors $v_i$ of $\tilde M_A$ by the process above. Since $m\times m$ matrices can have up to $m$ linearly independent eigenvectors in any field, $\tilde M_A$ has no eigenvectors linearly independent of these. That is, $\tilde M_A$ is fully diagonalisable in $\Zp$.

    The eigenvalues associated with eigenvector $v_i$ are $\big(|C^{(1)}|\theta(\chi_i(c^{(1)}_0)) d_i^{-1},\dots,|C^{(m)}|\theta(\chi_i(c^{(m)}_0)) d_i^{-1}\big)$. By the same argument as above, these lists of eigenvalues are linearly independent. In particular, no linear combination of more than one $v_i$ can be a joint eigenvector of all the $\tilde M_A$.
    
    In summary, all the $\tilde M_A$ are fully diagonalisable in $\Zp$ and they share all their eigenvectors. The joint eigenvectors $v_i$ are unique up to multiplication by a scalar and are of the form given in~\eqref{eq: dixon eigenvalue equation}.
\end{proof}

\begin{proof}[Proof of \Cref{thm: dixon character table}(5)] 
    From a straightforward calculation, $\pi(g^n) = \big[\prod_{i=1}^{n-1} \alpha(g^i,g)\big] \pi(g)^n$ for all $g\in G$ and $n\ge 0$. The trace of $\pi(g)^n$ is the sum of the $n$th powers of the eigenvalues of $\pi(g)$, that is, $\sum_{k=0}^e m_k(g) \zeta^{kn}$. Claim~(a) follows.

    Applying $\theta$ to this result and rearranging yields $\sum_{k=0}^e m_k(g) z^{kn} = \theta(\chi(g^n)) \prod_{i=1}^{n-1}\theta(\alpha(g^i,g))^{-1} \pmod p$. \Cref{eq: dixon fourier} follows by \Cref{thm: fourier}.

    By definition, $m_k(g)\ge 0$ for all $0\le k<e$ and $\sum_{k=0}^{e-1}m_k(g) = \dim\pi$ for all $g\in G$. Therefore, $m_k(g) \le \dim \pi$. In particular, if $\pi$ is irreducible, \Cref{thm: regular decomposition}(2) implies $\dim\pi \le \sqrt{|G|} < p$. That is, $0\le m_k(g)<p$ for all $g\in G$ and $0\le k< e$, i.e., they are the smallest nonnegative integers that satisfy~\eqref{eq: dixon fourier}.
\end{proof}

The last essential ingredient of the algorithm, not addressed by \Cref{thm: dixon character table}, is the ``normalisation'' of the eigenvectors. Indeed, in the notation of the proof above, $av_j$ for any $a\in\Zp$ is as good an eigenvector as $v_j$, but applying the discrete Fourier transform \eqref{eq: dixon fourier} to it would return $am_k\bmod p$ instead of the correct eigenvalue multiplicities $m_k$. 

We shall use~\eqref{eq: orthonormality mod p} to fix this overall scalar. For concreteness, let $\tilde v_j$ be a joint eigenvector of the $\tilde M_A$ corresponding to the projective irrep character $\chi_j$ such that for the class $C^{(1)}$ of the identity, $\tilde v_i^{(1)} = 1$. To simplify the notation, we define 
\begin{align*}
    \hat v_j(g) &= \begin{cases}
        \tilde v_j^{(i)} \theta\big(\beta\big(c^{(i)}_0,g\big)\big) & \text{$g$ is $\alpha$-regular in class $C^{(i)}$}\\
        0 & \text{$g$ is not $\alpha$-regular}
    \end{cases}\\
    \hat{\bar v}_j(g) &= \hat v_j (g^{-1}) \theta(\alpha(g,g^{-1})^{-1}).
\end{align*}
Since $v_j^{(1)} = d_j$, $\tilde v_j = d_j^{-1} v_j \pmod p$, and by \Cref{thm: character basics}(3), $\hat v_j(g) = d_j^{-1} \theta(\chi_j(g))\pmod p$. By \Cref{thm: character basics}(2) then, $\hat {\bar v}_j(g) = d_j^{-1} \theta\big(\overline{\chi_j(g)}\big)\pmod p$, whence, by~\eqref{eq: orthonormality mod p},
\begin{align}
    \mathcal N_j &:=\sum_{i=1}^m |C^{(i)}|\hat v_j\big(c^{(i)}_0\big) \hat {\bar v}_j\big(\big(c^{(i)}_0\big)\big) = d_j^{-2}
    \sum_{i=1}^m |C^{(i)}| \theta\big(\chi_j\big(c^{(i)}_0\big)\big)\theta\big(\overline\chi_j\big(c^{(i)}_0\big)\big)  =d_j^{-2}|G|.\pmod p
    \label{eq: normalisation mod p}
\end{align}
Given $\tilde v_j$, $\mathcal{N}_j$ can be computed explicitly, so our task is finding $d_j$ such that $\mathcal N_j d_j^2 = |G|\pmod p$. Since there exists an eigenvector $v_j$ that satisfies~\eqref{eq: orthonormality mod p}, this equation has a solution, and since the multiplication group of $\Zp$ is cyclic of even order,\footnote{$p\ge3$ unless $e=1$, which is trivial.} there are precisely two solutions, $\pm d_j\pmod p$. To select the correct solution, we note that \Cref{thm: regular decomposition}(2) implies $0<d_j < \sqrt{|G|}$ for all projective irreps (cf.~the proof of \Cref{thm: dixon character table}(5) above). Since we require $p>2\sqrt{|G|}$, only one of $\pm d_j \bmod p$ can satisfy this, so $d_j$ can be uniquely determined.

These arguments can be summarised in the following algorithm.
\begin{algorithm}[Eigenvector normalisation in $\Zp$]
    \label{alg: normalisation mod p}
    Given a joint eigenvector $\tilde v$ in $\Zp$ of the matrices $\tilde M_A$ defined in~\eqref{eq: dixon M matrix} such that $\tilde v^{(1)}=1$,
    \begin{enumerate}
        \item Compute $\mathcal N$ using~\eqref{eq: normalisation mod p}.
        \item Find the unique solution of $\mathcal N d^2 = |G|\pmod p$ such that $0<d< \sqrt{|G|}$. This $d$ is the degree of the irrep corresponding to the eigenvector $\tilde v$.
        \item The irrep characters are given by $v^{(i)} = \theta\big(\chi\big(c^{(i)}_0\big)\big) = d\tilde v^{(i)}\pmod p$ for all $1\le i\le m$.
    \end{enumerate}
\end{algorithm}

Finally, we note that the character of each $g\in G$ can be reconstructed without performing the full Fourier transform~\eqref{eq: dixon fourier}. In particular, let $n$ be the order of $g$ in $G$, and let $h$ be an integer such that $A = \prod_{i=1}^{n-1} \alpha(g^i, g) = \zeta^{nh}$; such an $h$ exists as the order of $g$ in $\alpha$ divides $e$. Then $\pi(g)^n = A\id_V$, so its eigenvalues are of the form $\zeta^{h + ke/n}$ for some $0<k\le n$. It follows from \Cref{thm: dixon character table}(5b) that the multiplicity of each of these is
\begin{equation*}
    \tilde m_k (g):= m_{h+ke/n} (g)= n^{-1}\sum_{j=0}^{n-1} (z^{e/n})^{-jk}  \times z^{-jh}\theta(\chi(g^j)) \prod_{i=1}^{j-1}\theta(\alpha(g^i,g))^{-1} \pmod p,
\end{equation*}
so it is enough to perform a discrete Fourier transform of size $n$ rather than $e$.

\section{Dixon's decomposition algorithm for projective representations}
\label{sec: dixon irrep matrix}

\Cref{thm: dixon irrep matrix}(1) hinges on the fact that $\langle f\rangle_\pi$ is a linear map between $\pi$ and itself, that is, it commutes with $\pi(g)$ for all $g\in G$. Therefore, we start by proving the statement for all such matrices in \Cref{thm: symmetry protected degeneracy}. \Cref{thm: dixon irrep matrix}(2) is, in a sense, the converse of \Cref{thm: schur lemma}, whence \Cref{thm: schur converse}.

\begin{lemma}
    \label{thm: symmetry protected degeneracy}
    Let $(\pi,V,\alpha)$ be a projective representation of $G$ and let $f:V\to V$ be a linear map that commutes with $\pi(g)$ for all $g\in G$. Then each eigenspace $W$ of $f$ is $G$-stable, that is, $\pi(g)W \subseteq W$ for all $g\in G$.
\end{lemma}
\begin{proof}
    Let $W$ be an eigenspace of $f$ and let $\lambda$ be the corresponding eigenvalue. Then if $f\cdot v =\lambda v$, then for all $g\in G$, $f\cdot\pi(g)\cdot v = \pi(g)\cdot v=\lambda \pi(g)\cdot v$, i.e., $\pi(g) \cdot v$ is also an eigenvector of $f$ with the same eigenvalue. That is, $v\in W\implies \pi(g)\cdot v\in W$ for all $g\in G$, as required.
\end{proof}

\begin{lemma}
    \label{thm: schur converse}
    Let $(\pi,V,\alpha)$ be a projective representation of $G$. Then $\pi$ is irreducible if and only if
    \begin{equation*}
        \langle f\rangle_\pi=\frac1{|G|}\sum_{g\in G} \pi(g)^{-1} f \pi(g) = \frac1{|G|}\sum_{g\in G} \pi(g) f \pi(g)^{-1}
    \end{equation*}
    is a multiple of the identity for all linear maps $f:V\to V$.
\end{lemma}
\begin{proof}
    The $\implies$ direction restates \Cref{thm: schur lemma} and \Cref{thm: schur shuffled}.
    To prove the $\impliedby$ direction, let $\pi$ be reducible, that is, let $V=V_1\oplus V_2$ such that $V_1$ and $V_2$ are invariant under $\pi(g)$ for all $g\in G$. Let $f$ be the linear map defined by $f(v)=v$ for all $v\in V_1$ and $f(v)=0$ for all $v\in V_2$. Then, since $V_1$ and $V_2$ are invariant under $\pi(g)$ for all $g\in G$,
    $\pi(g)^{-1}\cdot f\cdot \pi(g)\cdot v = v$ for all $v\in V_1$ and $\pi(g)^{-1}\cdot f\cdot \pi(g)\cdot v = 0$ for all $v\in V_2$. Therefore, $\pi(g)^{-1}f\pi(g) = f$ for all $g\in G$, so $\sum_{g\in G} \pi(g)^{-1} f \pi(g) = |G|f$, which is not a multiple of $\id_V$.
\end{proof}

\begin{proof}[Proof of \Cref{thm: dixon irrep matrix}]
    \begin{enumerate}
        \item $\langle f\rangle_\pi$ commutes with $\pi(g)$ for all $g\in G$:
        \begin{equation*}
            \pi(g)^{-1} \langle f\rangle_\pi \pi(g) = \frac1{|G|}\sum_{h\in G} \pi(g)^{-1}\pi(h)^{-1} f \pi(h)\pi(g) = \frac1{|G|}\sum_{h\in G} \pi(hg)^{-1} f \pi(hg) = \langle f\rangle_\pi,
        \end{equation*}
        so by \Cref{thm: symmetry protected degeneracy}, each eigenspace $W$ of it is invariant under $\pi$, so restricting $\pi$ to $W$ yields a valid projective representation.
        \item According to \Cref{thm: schur converse}, $\pi$ is irreducible if and only if $\langle f\rangle_\pi\propto \id_V$ for all linear maps $f:V\to V$. This obviously implies $\langle f_i\rangle_\pi\propto \id_V$ for all $1\le i\le (\dim V)^2$. The other direction holds since all linear maps $f:V\to V$ can be written as linear combinations of the $f_i$ and the map $f\mapsto\langle f\rangle_\pi$ is linear.
    \end{enumerate}
\end{proof}

\Cref{thm: dixon irrep matrix} does not directly yield an algorithm to generate the matrices of every irreducible projective representation. Instead, it gives us a way to systematically decompose any projective representation into irreps:
\begin{algorithm}[Irrep decomposition]
    \label{alg: dixon recursive}
    Given a projective representation $(\pi, V, \alpha)$ of $G$,
    \begin{enumerate}
        \item Find a basis $\{f_i\}_{i=1}^{(\dim V)^2}$ of linear maps $V\to V$.
        \item Find $1\le i\le (\dim V)^2$ such that $\langle f_i\rangle_\pi$ is not a multiple of the identity. Decompose $V$ into eigenspaces of $\langle f_i\rangle_\pi$ and repeat from (1) for each of them.
        \item If $\langle f_i\rangle_\pi$ is a multiply of the identity for all $1\le i\le (\dim V)^2$, $\pi$ is irreducible and we are done.
    \end{enumerate}
\end{algorithm}
Alternatively, we can take a random linear map $f:V\to V$, for which the eigenspaces of $\langle f_i\rangle$ are all irreducible with probability 1.

From here, we can take at least two different approaches to generate the matrices of every irreducible representation, which are the subject of the next two subsections.

\subsection{Decomposing the regular projective representation}

The first approach is to decompose a representation known to contain all irreps at least once. The classic example is the regular projective representation (\Cref{thm: regular decomposition}). In this special case, the space of all $\langle f\rangle_\pi$ is spanned by $|G|$ rather than $|G|^2$ linear maps, which are conveniently labelled with elements of $G$. In the usual basis of the regular representation, these are given by the relations 
\begin{equation*}
    \tilde L_h(e_g) :=  \alpha(g,h^{-1})e_{gh^{-1}}
\end{equation*}
for the left regular projective representation $L$ and by
\begin{equation*}
    \tilde R_h(e_g) :=\alpha(g^{-1},h^{-1}) e_{hg}
\end{equation*}
for the right regular projective representation $R$. As required, $[L(g),\tilde L_h] = [R(g),\tilde R_h]=0$ for all $g,h\in G$.

\subsection{Iterative construction}

Writing out matrices of size $|G|\times|G|$ may not be viable for large groups, even if their irreps have comparatively small dimensions. Indeed, Ref.~\cite{Dixon1970} does not take this approach, but rather uses the following result due to Burnside to iteratively find every irrep starting from a single faithful representation:
\begin{proposition}[\cite{Burnside1911}, \S226]
    \label{thm: faithful rep powers}
    Let $\rho$ be a faithful linear representation of $G$. Then every irreducible linear representation of $G$ is contained in at least one of $\rho,\rho\otimes\rho,\rho\otimes\rho\otimes\rho,\dots$ 
\end{proposition}

We can find all projective irreps of $G$ in a similar spirit. While $\pi\otimes\pi$ for a projective representation $(\pi,V,\alpha)$ has multiplier $\alpha^2$ rather than, the tensor product of $\pi$ with a linear representations has the right multiplier. Furthermore, we can transpose the full set of linear irreps into that of the projective irreps in the following way.

\begin{proposition}
    \label{thm: projective irreps from linear irreps}
    Let $(\pi,V,\alpha)$ be a projective representation of $G$. Let $\rho_1,\dots,\rho_n$ be the inequivalent linear representations of $G$. Then every irreducible projective representation of $G$ with multiplier $\alpha$ is contained in at least one of $\pi\otimes\rho_1,\dots,\pi\otimes\rho_n$.
\end{proposition}
\begin{proof}
    It follows from \Cref{thm: tensor product} that if $\rho$ is a linear representation, $\pi\otimes\rho$ is a projective representation with multiplier $\alpha$.
    Assume there is an irreducible projective representation $\tilde\pi$ that does not appear in any of $\pi\otimes\rho_1,\dots,\pi\otimes\rho_n$. Then $\tilde\pi$ cannot appear in the direct sum of any number of these either, including $\pi\otimes R$, where $R$ is the regular linear representation of $G$. The characters of $\pi\otimes R$ are $\chi_{\pi\otimes R}(1) = (\dim \pi) |G|$ and $\chi_{\pi\otimes R}(g)=0$ for all $1\neq g\in G$. This implies (cf. the proof of \Cref{thm: regular decomposition}) that $\pi\otimes R$ contains every projective irrep $\pi_i$ of $G$ with multiplicity $(\dim \pi)(\dim\pi_i)>0$. This contradiction concludes the proof.
\end{proof}

Alternatively, one can build all irreps from a single faithful projective representation and its dual:
\begin{proposition}
    Let $(\pi,V,\alpha)$ be a faithful projective representation of $G$. Then
    \begin{enumerate}
        \item $\pi\otimes \pi^*$ is a faithful linear representation of $G$.
        \item Every irreducible projective representation of $G$ with multiplier $\alpha$ is contained in at least one of $\pi,\pi\otimes\pi^*\otimes\pi,\pi\otimes\pi^*\otimes\pi\otimes\pi^*\otimes\pi\dots$
    \end{enumerate}
\end{proposition}
\begin{proof}
    \begin{enumerate}
        \item By \Cref{thm: dual} and \Cref{thm: tensor product}, $\pi\otimes\pi^*$ is a projective representation with multiplier $1$, i.e., a linear representation. Given that $\pi$ is faithful, $\pi\otimes\pi^*$ is trivially faithful too.
        \item By \Cref{thm: faithful rep powers}, every linear irrep $\rho$ of $G$ appears in $(\pi\otimes\pi^*)^{\otimes n}$ for some $n$, so $\rho\otimes\pi$ appears in $(\pi\otimes\pi^*)^{\otimes n}\otimes\pi$ for some $n$.
        By \Cref{thm: projective irreps from linear irreps}, every projective irrep of $G$ appears in $\rho\otimes\pi$ for some linear irrep $\rho$, hence in $(\pi\otimes\pi^*)^{\otimes n}\otimes\pi$ for some $n$.
    \end{enumerate}
\end{proof}

\end{document}